\documentclass[reqno,12pt]{amsproc}

\usepackage[top=3cm, bottom=3cm, left=2.54cm, right=2.54cm]{geometry}
\usepackage{amsmath,amsthm,amscd,amssymb,bbm,mathrsfs,latexsym,tikz,enumerate,verbatim,multirow}

\usepackage[colorlinks,citecolor=blue,pagebackref,hypertexnames=false]{hyperref}

\allowdisplaybreaks

\newcommand{\mR}{\mathbb{R}}

\newcommand{\WS}{\mathfrak{W}}

\newcommand{\rank}{{\mathop{\mathrm{rank}}}}
\newcommand{\nullity}{{\mathop{\mathrm{nullity}}}}
\def\dimsp(#1){\mathrm{dim} \left(\mathrm{span}(#1)\right)}

\def\Exp(#1){{\mathbb E}(#1)}

\def\inpr#1,#2{\t \hbox{\langle #1 , #2 \rangle} \t}
\def\ip<#1,#2>{\langle #1,#2 \rangle}

\def\norm#1{\left \Vert #1 \right \Vert}
\def\paren(#1){\left( #1 \right)}

\def\snorm#1{\Bigl \Vert #1 \Bigr \Vert}

\def\sparen(#1){\Bigl ( #1 \Bigr )}
\def\ssparen(#1){ (#1) }
\def\st{\thinspace : \thinspace}
\def\tri(#1,#2,#3){\begin{pmatrix} #1\\#2\\#3 \end{pmatrix}}

\def\v3(#1,#2,#3){\begin{pmatrix} #1 \\ #2 \\ #3 \end{pmatrix}}

\newcommand{\vecb}{{\boldsymbol{b}}}
\newcommand{\vecc}{{\boldsymbol{c}}}
\newcommand{\vece}{{\boldsymbol{e}}}
\newcommand{\vech}{{\boldsymbol{h}}}
\newcommand{\vecx}{{\boldsymbol{x}}}
\newcommand{\vecy}{{\boldsymbol{y}}}

\newcommand{\vecu}{{\boldsymbol{u}}}
\newcommand{\vecv}{{\boldsymbol{v}}}
\newcommand{\cof}{{\mathop{\mathrm{cof}}}}
\newcommand{\bxi}{\boldsymbol{\xi}}

\newcommand{\zero}{\boldsymbol{0}}
\newcommand{\one}{\boldsymbol{1}}
\newcommand{\hatX}{{\hat X}}
\newcommand{\hatB}{{\hat B}}
\newcommand{\hatG}{{\hat G}}
\newcommand{\hatx}{{\hat \vecx}}

\newcommand{\ds}{\displaystyle}

\numberwithin{equation}{section}
\newtheorem{theorem}{\bf Theorem}[section]
\newtheorem{lemma}[theorem]{\bf Lemma}
\newtheorem{corollary}[theorem]{\bf Corollary}
\newtheorem{proposition}[theorem]{\bf Proposition}

\theoremstyle{definition}
\newtheorem{definition}[theorem]{\bf Definition}

\newtheorem{example}[theorem]{\bf Example}

\theoremstyle{remark}
\newtheorem{remark}[theorem]{\bf Remark}
\numberwithin{equation}{section}

\begin{document}
\title{Beyond trees: the metric geometry of subsets of weighted Hamming cubes}

\author{Ian Doust} 
\address{School of Mathematics and Statistics, University of New South Wales, Sydney, New South Wales 2052, Australia}
\email{i.doust@unsw.edu.au}

\author{Anthony Weston}
\address{Arts and Sciences, Carnegie Mellon University in Qatar, Education City,
PO Box 24866, Doha, Qatar}
\email{aweston2@andrew.cmu.edu}



\begin{abstract}
Associated to any finite metric space are a large number of objects and quantities which provide some degree of structural or geometric information about the space. In this paper we show that in the setting of subsets of weighted Hamming cubes there are unexpected relationships between many of these quantities. We
obtain in particular formulas for the determinant of the distance matrix, the $M$-constant and the cofactor sum for such spaces.
In general, these types of results
offer valuable insights into the combinatorial optimization of certain constrained quadratic forms on finite metric spaces. A key focus in this context are embedding properties of negative type metrics, which play a prominent
role in addressing important questions like the sparsest cut problem in graph theory.
The current work extends previous results for unweighted metric trees, and more generally, for subsets of standard Hamming cubes, as well as results for weighted metric trees. Finally
we consider polygonal equalities in these spaces, giving a
complete description of the nontrivial $1$-polygonal equalities
that can arise in weighted Hamming cubes. 
\end{abstract}
\maketitle


\section{Introduction}

It is now over half a century since Graham and Pollak's surprising discovery \cite{GP} that the determinant of the distance matrix for an unweighted metric tree on $n+1$ vertices is always $(-1)^n n 2^{n-1}$, independent of the actual structure of the tree. Since then many authors have explored how this fact might be generalized. One natural direction is to look at what happens if one adds weights to the edges of the tree. On the other hand, since every unweighted tree can be isometrically embedded in the vertices of a Hamming cube with the Hamming metric, one might also look at more general subsets of such
cubes. This current paper combines and extends the work in
both of these directions. In a completely new line of enquiry,
we examine and classify the nontrivial $1$-polygonal
equalities that can arise in weighted Hamming cubes.
We then use this classification to provide a surprising
link between the rank of the distance matrix and the
linear algebraic properties of subsets of these spaces.

Apart from the determinant of the distance matrix $D_{T}$,
a number of other quantities have been found to dependent only on $n$, the number of edges of such a tree, and not on the tree structure. 
\begin{enumerate}[(i)]
  \item In \cite[Theorem~4.1]{DRSW} it was shown
  that the sum of the entries in the inverse of $D_T$ is
  $\frac{2}{n}$.
  \item Graham, Hoffman and Hosoya \cite{GHH} showed that the cofactor sum $\cof(D_T)$ equals $(-2)^n$.
  \item Doust and Wolf \cite{DWo} showed that the $M$-constant of the tree is always $\frac{n}{2}$.
\end{enumerate}
Some of these quantities are closely connected to the metric geometry of these spaces, particularly the embedding notions
of negative type and generalized roundness.
Underlying many of these ideas is what Linial \cite{Lin} termed
the geometrization of combinatorics; the idea that information
about combinatorial objects can be gleaned from looking at them
from a geometric perspective.
In Section~\ref{S:Background} we shall recall the definitions of the aforementioned concepts as well as some of the previously
known extensions of the above results (i) -- (iii).

The setting for the current paper is what we shall call weighted Hamming cubes.
Given a list of positive weights $\WS = [w_1,\dots,w_n]$, we let
  \[ H_\WS = \prod_{k=1}^n \{0,w_k\} \subseteq \mR^n, \]
which we shall consider as a  metric subspace of $\mR^n$ equipped with the $\ell^1$ metric.
It is easy to verify that any $(n+1)$-point weighted tree can be isometrically embedded (algorithmically) in such a cube and that the embedded points form an affinely independent subset of $\mR^n$. 

A major component in developing the theory here will be the following extension of results of Doust and Weston \cite{DWe},
Murugan \cite{Mur}, Kelleher, Miller, Osborn and Weston \cite{KMOW}, and Doust, Robertson, Stoneham and Weston \cite{DRSW} for Hamming cubes.

\begin{theorem}\label{omnibus} Suppose that $(X,d)$ 
is a metric subspace of $(H_\WS,\norm{\cdot}_1)$. Then the following are equivalent:
\begin{enumerate}[(1)]
  \item $X$ is affinely independent.
  \item $D_X$ is invertible.
  \item $(X,d)$ is of strict $1$-negative type.
  \item The supremal $p$-negative type of $(X,d)$ is greater than $1$.
  \item The maximal generalized roundness of $(X,d)$ is greater than $1$.
  \item The $1$-negative type gap $\Gamma_1(X)$
  of $(X, d)$ is positive.
  \item There are no nontrivial $1$-polygonal equalities in $(X,d)$.
\end{enumerate}
\end{theorem}

Conditions (3), (4), (5) and (6) are known to be equivalent in the category of finite metric spaces. In Section~\ref{S:HammingCubes} we shall show that (1) and (2) are equivalent for subsets of $H_\WS$. We also give a formula for $\det(D_X)$ which generalizes the Graham--Pollak formula. 
In Section~\ref{S:Aff-strict} we show that (1) and (3) are also equivalent for subsets of $H_\WS$. 

The equivalence of (2) and (3) in fact holds for any finite subsets of $(\mR^n,\norm{\cdot}_1)$ (see Theorem~\ref{DXInv-iff-strict}), but we note in Example~\ref{Inv-not-strict} that while in the generality of finite metric spaces (3) always implies (2), the converse implication fails.

In Section \ref{S:M-const},
extending the known results for weighted trees and subsets of Hamming cubes, we given formulas for the $M$-constant, the
sum of the entries of $D_X^{-1}$ and the cofactor sum of $D_X$
in weighted Hamming cubes. 

Kelleher et al.~\cite[Lemma 3.21]{KMOW} showed that (3) and (7) are equivalent for  spaces of $1$-negative type, and so certainly for finite subsets of  $(\mR^n,\norm{\cdot}_1)$. More recently the authors \cite{DWe2} have shown that this equivalence in fact holds in any metric space. In Section~\ref{S:Poly-ineq} we shall examine nontrivial $1$-polygonal equalities in Hamming cubes,
and give a direct proof that conditions (1) and (7) are equivalent for weighted Hamming cubes.
Our proof of this equivalence depends, in an essential way, on
developing a structural description of the nontrivial
$2$-polygonal equalities admitted in inner product spaces.
We then proceed to give a complete description of the nontrivial
$1$-polygonal equalities admitted in weighted Hamming cubes. This description is then applied in Theorem~\ref{rank-thm} show that the dimension of the distance matrix for a subset of a weighted Hamming cube is completely determined by the dimension of the affine subspace containing the set. 

Several authors have studied classes of graphs which can be isometrically embedded in Hamming cubes (see, for example, \cite{GW,RoW,BSetc}). 
Not unlike the main result in Roth and Winkler
\cite[Section 3]{RoW}, one may view Theorem \ref{omnibus}
(taken in its entirety) as a `collapse of the metric hierarchy'
for weighted Hamming cubes.

At the end of the paper we examine more fully the extent to which the setting of weighted Hamming cubes is necessary for the various implications contained in Theorem~\ref{omnibus}.
For instance, we give an example of a finite subset of $(\mR^2,\norm{\cdot}_1)$ which satisfies conditions (2) and (3) but not condition (1). A number  of  open problems
are also discussed.

\section{Background and preliminary
observations}\label{S:Background}

Given a finite metric space $(X,d)$ with elements $x_0,x_1,\dots,x_m$ we shall denote the
\textbf{distance matrix} of $X$ by $D_X = (d(x_i,x_j))_{i,j}$. Given $p > 0$, we shall let $D_{X,p}$, or simply $D_p$, denote the \textbf{$p$-distance matrix} of $X$,
$D_p = (d(x_i,x_i)^p)_{i,j}$.
Although the concepts which we shall now introduce are applicable for general metric spaces, to simplify the definitions and exposition, we shall restrict ourselves to the setting of finite metric spaces. Further, we shall always assume that the spaces are nontrivial, in the sense that they have at least two elements.

The usual inner product in $\mR^n$ will be denoted $\ip<\cdot,\cdot>$. 
Throughout we shall use $\one$ to denote a vector of an appropriate size whose entries are all $1$. Thus, if $A$ is any $n \times n$ matrix, the sum of the entries of $A$ can be written as $\ip<A\one,\one>$. The \textbf{cofactor sum} of $A$ is the quantity
  \[ \cof(A) = \sum_{i,j=1}^n (-1)^{i+j} \det(A_{i|j}) \]
where $A_{i|j}$ is the submatrix of $A$ obtained by deleting the $i$th row and the $j$th column. 

\subsection{$p$-negative type and generalized roundness}

The concept of $p$-negative type has its roots in embedding questions. 

Let $F_0$ denote the hyperplane of appropriate dimension of all vectors $\bxi$ such that $\ip<\bxi,\one> = 0$; that is, the set of all column vectors whose components sum to zero.

\begin{definition} Let $(X,d)$ be a finite metric space with $m+1$ points $\{x_{0}, x_{1}, \ldots, x_{m} \}$
and suppose that $p \ge 0$.
We say that $(X,d)$ is of \textbf{$p$-negative type} if
  \begin{equation}\label{p-neg}
    \ip<D_p \bxi,\bxi> = \sum_{i,j} d(x_i,x_j)^p \xi_i \xi_j \le 0
  \end{equation}
for all $\bxi = (\xi_0,\dots,\xi_{m})^{T} \in F_0$.  

If strict inequality holds in (\ref{p-neg}) except when $\bxi = \zero$ we say that $(X,d)$ is of \textbf{strict $p$-negative type}.
\end{definition}

The classical result of Schoenberg \cite{Sc3} is that $(X,d)$ may be isometrically embedded in a Hilbert space if and only if it is of $2$-negative type.  Spaces of $1$-negative type, often simply referred to as spaces of negative type (or quasihypermetric spaces; see \cite{NIC}), have arisen is many areas. An important example is in the study of the sparsest cut problem in graph theory. See, for instance, \cite{CGR,ALN}.

Of particular importance in the current study are the following facts due to
Schoenberg \cite{Sc3} and Hjorth et al.\ \cite{HLMT},
respectively.
\begin{itemize}
  \item[(i)] Every subset of $\ell^1$ is of $1$-negative type.
  \item[(ii)] Every finite metric tree is of strict
  $1$-negative type.
\end{itemize}

Doust and Weston \cite{DWe} introduced an enhanced measure of the $p$-negative type inequality (\ref{p-neg}). The
\textbf{$p$-negative type gap} for a finite metric space
$(X, d)$ of $p$-negative type
is the largest non-negative constant $\Gamma_p(X)$ such that 
  \begin{eqnarray}\label{p-gap}
  \Gamma_p(X) \norm{\bxi}_1^2
       + 2 \ip<D_p \bxi,\bxi> & \le & 0
  \end{eqnarray}
for all $\bxi \in F_0$. The main result of \cite{DWe} was that if $(X,d)$ is an $(m+1)$-point metric tree with edge weights $[w_1,\dots,w_m]$, then
  \begin{equation}\label{1-gap}
     \Gamma_1(X) = \left(\sum_{k=1}^m w_k^{-1}\right)^{-1}, 
  \end{equation}
which again is independent of the structure of the
tree and independent of which weight is assigned to which edge.

The concept of generalized roundness was introduced by Enflo
\cite{En1} in the 1960s as a tool to study universal uniform
embedding spaces.

\begin{definition} Let $(X,d)$ be a metric space and suppose that $p \ge 0$. We say that $(X,d)$ has \textbf{generalized roundness $p$} if
  \begin{equation}\label{gen-round}
  \sum_{i < j} d(a_i,a_j)^p + \sum_{i < j} d(b_i,b_j)^p
     \le \sum_{i,j} d(a_i,b_j)^p
  \end{equation}
for all finite lists of elements $[a_1,\dots,a_N]$ and $[b_1,\dots,b_N]$ of $X$.
\end{definition}

It is worth stressing that the elements of the lists in this definition do not need to be distinct. One can however assume that the two lists are disjoint.  

It was only much later that Lennard, Tonge and Weston \cite{LTW} showed that the notions of $p$-negative type and generalized
roundness $p$ are in fact equivalent in any metric space.

\begin{theorem}\label{intervals}
Let $(X,d)$ be a metric space and suppose
$p \ge 0$. Then:
\begin{enumerate}[(i)]
  \item $(X,d)$ is of $p$-negative type if and only if $(X,d)$ has generalized roundness $p$.
  \item The set of all values $p$ for which $(X,d)$ has
  generalized roundness $p$ is always an interval of the form
  $[0, \wp]$, with $\wp \geq 0$, or $[0, \infty)$.
\end{enumerate}
\end{theorem}

It is important to note in the statement of Theorem
\ref{intervals} (ii) that it is possible that $\wp = 0$. In
other words, there exist (necessarily infinite) metric spaces
that do not have generalized roundness $p$ for any $p > 0$.
The first such example was constructed by Enflo \cite{En1}.

Given a metric space $(X,d)$ we let
  \[ \wp(X,d) = \sup\{p \st \text{$(X,d)$ has generalized roundness $p$}\}. \]
In settings where the metric on $X$ is entirely clear, we may
write $\wp(X)$ instead of $\wp(X,d)$.
By Theorem \ref{intervals} (ii), if $\wp(X,d)$ is finite, then
the underlying supremum is clearly attained.
Depending on the context, $\wp(X,d)$ is known as either the \textbf{maximal generalized roundness}, or the \textbf{supremal
$p$-negative type} of $(X,d)$. S{\'a}nchez \cite{San}, following work of Wolf \cite{Wo}, gave an important characterization of $\wp(X,d)$ which links this value to two other quantities which will appear in this study.

\begin{theorem}\label{Sanchez}
Suppose that $(X,d)$ is a finite metric space and that $\wp(X,d)$ is finite. Then
  \[ \wp(X,d) = \min\{ p \ge 0 \st \det(D_p) = 0
  \text{ {\rm or} }
  \ip<D_p^{-1} \one,\one> = 0 \}. \]
\end{theorem}

In relation to Theorem \ref{Sanchez} we note
that for any
metric space $(X, d)$, finite or otherwise,
$\wp(X,d) = \infty$ if and only if $d$ is an ultrametric. This result
is due to Faver et al.\ \cite{Fav}.

Li and Weston \cite{LiW} showed that for any finite metric space, 
  \begin{equation}\label{LW-str} 
    \{p \st \text{$(X,d)$ is of strict $p$-negative type}\}
   = [0,\wp(X,d)) . 
   \end{equation} 

It follows immediately from S{\'a}nchez's formula that if $(X,d)$ is a finite metric space of strict $1$-negative type, then $D_X$ is invertible. In general, the converse implication fails.

\begin{example}\label{Inv-not-strict}
Let $(X,d)$ be a 5 point metric space with distance matrix
  \[ D_X = \begin{pmatrix}
         0&\alpha&\alpha&1&1\\ 
         \alpha&0&\alpha&1&1\\ 
         \alpha&\alpha&0&1&1\\ 
         1&1&1&0&\alpha\\ 
         1&1&1&\alpha&0
        \end{pmatrix}
  \]
with $\alpha = 12/7$. 
A calculation shows that 
  \[ \det(D_p) = 6 \alpha^{3p} - 2 \alpha^{5p}. \]
This has a first positive zero at 
  \[ p_0 = \frac{\log 3}{2 \log \alpha} > 1.  \]
On the other hand,
  \[ \ip<D_p^{-1} \one,\one> = \frac{7\alpha^p-12}{2\big(\alpha^{2p}-3\bigr)}, \]
which clearly has a single zero at $p_1 = 1$. Thus $(X,d)$ has supremal $p$-negative type $1$ (and hence is not is of strict $1$-negative type), but
$D_X$ is invertible.
\end{example}

However, as we shall see in the next subsection,
this behaviour cannot occur
in the setting of subsets of $(\mR^n,\norm{\cdot}_1)$. 

\subsection{The $M$-constant} 

The $M$-constant of a compact metric space was introduced by Alexander and Stolarsky \cite{AS}, and further investigated by Nickolas and Wolf (see \cite{NIC}), who showed that it contains useful information about embedding of the space into Euclidean space. 
If $(X, d) = (\{x_0,x_1,\ldots,x_m\}, d)$ is a
finite metric space with distance matrix $D_X$, then
  \[ M(X) = \sup_{\bxi \in F_1}
  \ip<D_X \bxi, \bxi> \in (0,\infty], \]
where $F_1 = \{\bxi \in \mR^{m+1} \st \ip<\bxi,\one> = 1\}$.

We recall here a number of results which link this quantity to the concept of strict $1$-negative type and to properties of the distance matrix $D_X$.

\begin{proposition}\label{DWo3.2} \cite[Theorem 3.2]{DWo}
Suppose that $(X,d)$ is an $(m+1)$-point metric space of $1$-negative type with distance matrix $D_X$. Then there exists $\vecb \in \mR^{m+1}$ such that $D_X \vecb = \one$ and $\ip<\vecb,\one> \ge 0$. Further
\begin{enumerate}[(i)]
\item The value of $\ip<\vecb,\one>$ is independent of $\vecb$.
\item $M(X) < \infty$ if and only if $\ip<\vecb,\one> > 0$. In this case
   $  M(X) = \ip<\vecb,\one>^{-1} $.
\end{enumerate}
\end{proposition}

\begin{proposition}\label{DInv-iff-M_finite} \cite[Theorem 3.3]{DWo}
Suppose that $(X,d)$ is a finite metric space of strict $1$-negative type with distance matrix $D_X$. Then 
\begin{enumerate}[(i)]
  \item $D_X$ is invertible.
  \item $M(X) < \infty$.
  \item $\ds M(X) = \frac{1}{\ip<D_X^{-1}\one,\one>} $.
\end{enumerate}
\end{proposition}

In general $M(X)$ can be infinite, but as observed by Nickolas and Wolf \cite{NWF}, this can not happen for finite subsets of $(\mR^n,\norm{\cdot}_1)$.

\begin{proposition}\label{NW4.7}\cite[Theorem~4.7]{NWF}
If $X$ is an $m+1$ point subset of $(\mR^n,\norm{\cdot}_1)$,
then $M(X) \le \frac{m+1}{4} \mathop{\mathrm{diam}}(X)$.
\end{proposition}

\begin{theorem}\label{DXInv-iff-strict}
Suppose that $(X,d)$ is a finite 
subset of $(\mR^n,\norm{\cdot}_1)$ with distance matrix $D_X$. Then $D_X$ is invertible if and only if $(X,d)$ is of strict
$1$-negative type. 
\end{theorem}

\begin{proof}  To prove the nontrivial implication, suppose that $D_X$ is invertible, but that $(X,d)$ is not of strict $1$-negative type. 
Since $(X,d) \subseteq (\mR^n,\norm{\cdot}_1)$, it is certainly of $1$-negative type, we must have $\wp(X,d) = 1$. Theorem~\ref{Sanchez} then implies that $\ip<D_X^{-1} \one,\one> = 0$. 
Taking $\vecb = D_X^{-1}\one$ in Proposition~\ref{DWo3.2}, this implies that $M(X) = \infty$. But by Proposition~\ref{NW4.7}, $M(X)$ is finite, which gives the desired contradiction.
\end{proof}

We can combine this with a result of Kelleher et al.~\cite{KMOW} to show that some of the implications from Theorem~\ref{omnibus} hold in the generality of finite subsets of $(\mR^n, \norm{\cdot}_1)$. In Section \ref{S:Examples} we will
show that the converse of the following theorem does not
hold.

\begin{theorem}\label{Aff-impl-strict+DXInv}
Suppose that $(X,d)$ is a finite 
subset of $(\mR^n,\norm{\cdot}_1)$ with distance matrix $D_X$. If $X$ is affinely independent then $D_X$ is invertible and $(X,d)$ is of strict $1$-negative type.
\end{theorem}

\begin{proof}
    The fact that affine independence implies that $(X,d)$ is of strict $1$-negative type is a special case of \cite[Corollary 4.10]{KMOW}. The result now follows from Theorem~\ref{DXInv-iff-strict}.
\end{proof}

\subsection{Polygonal equalities}

In studying a question of Bernius and Blanchard, Elsner et al.\
\cite{{EHKMZ}} characterized when one obtains the `polygonal equality'
  \begin{eqnarray}\label{Elsner}
  \sum_{i< j} \norm{\vecx_i - \vecx_j}_1 + \sum_{i< j} \norm{\vecy_i - \vecy_j}_1 & = & \sum_{i,j} \norm{\vecx_i - \vecy_j}_1
  \end{eqnarray}
for finite subsets $\{\vecx_i\}_{i = 1}^n$ and
$\{\vecy_j\}_{j = 1}^n$
of $L_1(\Omega,\mu)$. 

Motivated by this and some of the equivalent formulations of generalized roundness,
Kelleher et al.\ \cite[Definition 3.20]{KMOW} introduced the
concept of a nontrivial $p$-polygonal equality in a
general metric space.

\begin{definition}
Suppose that $(X,d)$ is a metric space and that $p \ge 0$. A \textbf{$p$-polygonal equality} in $(X,d)$ is an equality of the form
  \begin{equation}\label{p-poly} 
  \sum_{i=1}^s \sum_{j=1}^t m_i n_j d(x_i,y_j)^p
    = \sum_{1 \le i_1<i_2 \le s} m_{i_1} m_{i_2} d(x_{i_1},x_{i_2})^p
     + \sum_{1 \le j_1<j_2 \le t} n_{j_1} n_{j_2} d(y_{j_1},y_{j_2})^p
  \end{equation}
for some (not necessarily distinct) $x_1,\dots,x_s,y_1,\dots,y_t \in X$ and real weights $m_1,\dots,m_s$, $n_1,\dots,n_t$ with $\sum_{i=1}^s m_i = \sum_{i=j}^t n_j$.
\end{definition}

In order for this concept to be useful, one needs to eliminate `trivial' examples of such equalities. For the purpose of considering condition~(7) in Theorem~\ref{omnibus} it is not necessary that we give the original definition of a
\textbf{nontrivial $p$-polygonal equality}.
(The interested reader is referred to \cite[Definition 3.20]{KMOW}
or \cite[Theorem~2.3]{DWe2}) for details.) Rather we shall work with the
following characterization which applies in the category
of finite metric spaces.

\begin{theorem}\label{p-poly-equiv-form} 
Let $(X,d)$ be a finite metric space and suppose that $p \ge 0$. Then the following are equivalent.
\begin{enumerate}[(i)]
  \item $(X,d)$ admits a nontrivial $p$-polygonal equality.
  \item $\ip<D_p \bxi,\bxi> = 0$ for some nonzero $\bxi \in F_0$.
\end{enumerate}
\end{theorem}

\begin{remark}\label{ntpp}
Every nontrivial $p$-polygonal equality in the sense of
\cite[Definition 3.20]{KMOW} can be algorithmically
reduced to an identity of the form $\ip<D_{X,p} \bxi, \bxi> = 0$
and this may therefore be considered the standard form for
the polygonal equality. It is sufficient therefore in order
to understand all of the nontrivial $p$-polygonal equalities
in a space to provide a description of all such identities
$\ip<D_{X,p} \bxi, \bxi> = 0$ which are admitted in the space.
\end{remark}

An implicit consideration of nontrivial $p$-polygonal
equalities (without naming them) was already prominent
in Li and Weston \cite{LiW}. Subsequently, Weston \cite{W2017}
classified the nontrivial $p$-polygonal equalities in certain
subsets of $L_p$-spaces and used this to show that any
two-valued Schauder basis of $L_p$ must have strict
$p$-negative type. 

Building off of work in \cite{KMOW}, the authors have
recently shown that conditions (3) and (7) of Theorem
\ref{omnibus} are equivalent in all metric spaces.

\begin{theorem}\label{3=7} \cite[Theorem~1.3]{DWe2}
    Suppose that $(X,d)$ is a metric space. Then $(X,d)$ is of strict $p$-negative type if and only if $(X,d)$ does not admit any nontrivial $p$-polygonal equalities.
\end{theorem}

From (\ref{LW-str}) we immediately get that if $(X,d)$ is a finite metric space with $\wp(X)$ finite, then 
  \[  \{p > 0 \st \text{$(X,d)$ admits a nontrivial $p$-polygonal equality}\} = [\wp(X),\infty).\]

It is worth noting that if $(X,d)$ is a finite subset of $(\mR^n,\norm{\cdot}_1)$ with $\wp(X) = 1$ then it is easy to find nontrivial $1$-polygonal equalities. In this case $(X,d)$ is not of strict $1$-negative type, so by Theorem~\ref{DXInv-iff-strict}, $D_X$ is not invertible. S{\'a}nchez \cite[Corollary 2.5]{San} showed that then $\ker(D_X) \subseteq F_0$ and so any nonzero $\bxi \in \ker(D_X)$ gives a nontrivial $1$-polygonal equality.

\subsection{Known results for weighted trees}

The literature contains several formulas for the quantities $\det(D_X)$ 
and $\cof(D_X)$ for weighted trees, extending the results noted in the introduction.

Bapat, Kirkland and Neumann \cite{BKN} extended Graham and Pollak's result on the determinant of the distance matrix, giving a formula which is again independent of the structure of the tree. Related work can be found in \cite{ZD}.

\begin{theorem}\label{det-w-tree} \cite[Corollary 2.5]{BKN}
Suppose that $T$ is an $(n+1)$-point metric tree with edge weights $w_1,\dots,w_n$. Then 
  \begin{equation} 
  \det(D_T) =  (-1)^{n}\, 2^{n-1} \Bigl(\prod_{i=1}^n w_i\Bigr)
                     \Bigl( \sum_{i=1}^n w_i\Bigr).
  \end{equation}
\end{theorem}

Choudury and Khare \cite{CK} (and earlier authors) have looked at a rather more complicated setup which involves $q$-versions with bi-directed trees in which the matrices have entries $$\frac{q^{d(\vecx_i,\vecx_j)}-1}{q-1}.$$
They obtain somewhat more complex expressions for
$\det(D_T)$ and $\cof(D_T)$ in this setting.

%
%

\section{Weighted and unweighted Hamming cubes}\label{S:HammingCubes}

Let $H_n$ denote the usual (unweighted) Hamming cube, $H_n = \{0,1\}^n \subseteq \mR^n$, equipped with the $\ell^1$ metric. Given a weight list $\WS = [w_1,\dots,w_n]$, the diagonal matrix $W = \mathrm{diag}(w_1,\dots,w_n)$ acts as an invertible map from $H_n$ onto $H_\WS$. Given $X \subseteq H_\WS$, let $\hatX = W^{-1} X$ be the inverse image of $X$ in $H_n$. Clearly $\hatX$ is affinely independent if and only if $X$ is. 

Any $\vecx \in H_\WS$ can be written as 
  \[ \vecx = \sum_{k=1}^n w_k \, b_k \, \vece_k \]
where $\vece_k$ is the $k$th standard basis vector and $b_k \in \{0,1\}$. We shall denote the inverse image of $\vecx$ under $W$ by $\hatx$. That is, $\hatx = \sum_{k=1}^n b_k \, \vece_k \in H_n$. 

It is worth observing that for all $k$, the reflection map under which $w_k\, b_k\, \vece_k$, the $k$th coordinate of $\vecx$, is replaced by $w_k\, (1-b_k) \,\vece_k$ is an isometry of $H_\WS$. In particular applying one or more of these reflections does not affect the distance matrix of a subset $X \subseteq H_\WS$, nor whether the set is affinely independent. We may therefore, when convenient, assume that any subset $X = \{\vecx_0,\dots,\vecx_m\} \subseteq H_\WS$ has $\vecx_0 = \zero$.

In what follows we shall make use of a weighted inner product. If $\vecx = \sum_k w_k \, b_k \vece_k$ and $\vecy = \sum_k w_k \, b'_k \vece_k$, then we define
  \[ \ip<\vecx,\vecy>_W = \sum_{k=1}^n w_k\, b_k\, b_k'.\]
For all $\vecx \in H_\WS$ we have $\norm{\vecx}_1 = \ip<\vecx,\vecx>_W$.

Our first step is to show that $D_X$ is invertible if and only if $D_{\hatX}$ is invertible. 

Suppose then that $X = \{\vecx_0 = \zero,\vecx_1,\dots,\vecx_m\} \subseteq H_\WS$. Let $\vecu_X = (\norm{\vecx_1}_1,\dots,\norm{\vecx_m}_1)^T$ be the vector of norms of the nonzero elements of $X$. The above inner product determines a Gram matrix for $X$:
  \[ G_W  = \bigl( \ip<\vecx_i,\vecx_j>_W \bigr)_{i,j=1}^m. \]

Lemma~2.1 from \cite{DRSW} extends to the weighted setting.

\begin{lemma}\label{detDX}
Let $X = \{\vecx_0=\zero,\vecx_1,\dots,\vecx_m\}$ be a subset of $H_\WS$. Then
  \[ \det(D_X) = (-1)^{m-1} 2^{m-1} 
    \det \begin{pmatrix}
        0      & \vecu_X^T \\
        \vecu_X  & G_W
    \end{pmatrix}. \]
\end{lemma}

\begin{proof} 
Beginning with the distance matrix $D_X$, perform the elementary rows operations of first subtracting the first row from each of the other rows, and then subtracting the first column of that matrix from each of the other columns. 

If we write $\vecx_i = \sum_{k=1}^n w_k\, b_{i,k}\, \vece_k$, then a small calculation shows that the $(i,j)$th entry of the resulting matrix $D_1$ is
  \[ \begin{cases}
      \norm{\vecx_{i-1}}_1,  & \text{if $j = 0$}, \\
      \norm{\vecx_{j-1}}_1,  & \text{if $i = 0$}, \\
      \sum_{k=1}^n w_k \bigl( |b_{i-1,k}-b_{j-1,k}|- b_{i-1,k}-b_{j-1,k}\bigr),  &
                  \text{if $1 \le i,j \le m$.}
  \end{cases} 
  \]
Since each $b_{i,k} \in \{0,1\}$ one can readily verify that 
  \[  w_k \bigl(|b_{i-1,k}-b_{j-1,k}|- b_{i-1,k}-b_{j-1,k} \bigr)
       = -2 w_k\, b_{i-1,k} \, b_{j-1,k} \]
and so 
  \[ \sum_{k=1}^n w_k \bigl( |b_{i-1,k}-b_{j-1,k}|- b_{i-1,k}-b_{j-1,k}\bigr)
       = -2 \ip<\vecx_{i-1},\vecx_{j-1}>_W.\]
Thus 
  \[ D_1 = \begin{pmatrix}
      0   & \vecu_X^T \\
      \vecu_X  & -2\, G_W
  \end{pmatrix}\]
and so 
  \[ \det(D_X) = \det(D_1) =  (-2)^{m-1}
  \det \begin{pmatrix}
      0   & \vecu_X^T \\
      \vecu_X  & G_W
  \end{pmatrix}. \]
\end{proof}

A standard determinant identity (see, for example, \cite[Lemma~2.3]{DRSW}) shows that if $G_W$ is invertible then
  \begin{equation}\label{det-ident}
   \det \begin{pmatrix}
        0      & \vecu_X^T \\
        \vecu_X  & G_W
    \end{pmatrix}
    = \det(G_W) \det(-\vecu_X^T G_W^{-1} \vecu_X)
    = -\det(G_W) \ip<G_W^{-1} \vecu_X,\vecu_X>.
    \end{equation}

Given an $m \times n$ matrix $A$ and a $m$-element subset $S \subseteq \{1,\dots,n\}$, let $A_S$ denote the $m \times m$ submatrix of $A$ obtained by selecting the columns with indices in $S$. The following useful fact is a consequence of the Cauchy--Binet formula.

\begin{lemma}\label{CB-lem}
Let $A$ be an $m \times n$ matrix with $m \le n$. Then 
  \[ \det(A A^T) = \sum_{|S|=m} \det(A_S)^2 \]
where the sum is taken over all $m$ element subsets of $\{1,\dots,n\}$.
\end{lemma}

Suppose now that $m \le n$.
Let $B \in M_{mn}$ be the matrix whose $i$th row is $\vecx_i$. Then
  \[ G_W = B W^{-1} B^T = (BW^{-1/2})(BW^{-1/2})^T. \]
It follows then from Lemma~\ref{CB-lem} that
  \begin{equation}\label{CBin} \det(G_W) = 
   \sum_{|S|=m} \det \left((BW^{-1/2})_S\right)^2
  = \sum_{|S|=m} \frac{\det(B_S)^2}{\prod_{j \in S} w_j}.
  \end{equation}
Thus, $\det(G_W) = 0$ if and only if $\det(B_S) = 0$ for all $S$.

Let ${\hat B}$ be the $m \times n$ matrix whose $i$th row is $\hatx_i := W^{-1} \vecx_i \in H_n$.  That is, the $(i,k)$th  entry of $\hat B$ is $b_{i,k}$, and $B = {\hat B}W$. Let
  \[ {\hat G} = {\hat B}{\hat B}^T = \bigl( \ip<\hatx_i,\hatx_j>\bigr)_{i,j=1}^m
     = \Bigl(\sum_{k=1}^n b_{i,k} \, b_{j,k}\Bigr)_{i,j=1}^m. \]
be the corresponding Gram matrix. By Lemma~\ref{CB-lem} 
  \[ \det(\hatG) = \sum_{|S|=m} \det({\hat B}_S)^2 .\]

\begin{lemma}\label{hatG-lem} Suppose that $\{\vecx_0 = \zero,\vecx_1, \ldots, \vecx_m\} \subseteq H_\WS \subseteq \mR^n$ with $m \le n$
With the above notation:
\begin{enumerate}[(i)]
\item If $|S| = m$, then $\det(B_S) = 0$ if and only if $\det({\hat B}_S) = 0$.
\item The following are equivalent.
  \begin{enumerate}[(a)]
  \item $\det(G_W) = 0 $.
  \item $\det(\hatG) = 0$.   
  \item $\det(B_S) = 0$ for all $m$-element sets $S$.
  \item $\{\vecx_1,\dots,\vecx_m\}$ is linearly dependent.
  \item $\{\hatx_1,\dots,\hatx_m\}$ is linearly dependent.
  \end{enumerate}
\end{enumerate}
\end{lemma}

\begin{proof} 
(i) Let $S = [j_1,\dots,j_m]$ and let $W_{(S)} = \text{diag}(w_{j_1},\dots,w_{j_m})$, which is always invertible. Since $B_S = {\hat B}_S W_{(S)}$ this gives the desired conclusion. 

(ii) That (b) is equivalent to (e) is a standard property of the Gram matrix. We have already noted that (d) and (e) are equivalent. The remaining equivalences follow from (i), equation (\ref{CBin}) and Lemma~\ref{CB-lem}.
\end{proof}

Let $X = \{\vecx_0 = \zero,\vecx_1, \ldots, \vecx_m\} \subseteq H_\WS$. Then $\{\vecx_1,\dots,\vecx_m\}$ is linearly independent if and only if $\{\hatx_1,\dots,\hatx_m\}$ is linearly independent.

\begin{theorem}\label{AI-iff-invert}
Let $X = \{\vecx_0,\vecx_1, \ldots, \vecx_m\} \subseteq H_\WS$.
Then $X$ is affinely independent if and only if $\det(D_X) \ne 0$.
\end{theorem}

\begin{proof}
The forward implication follows from Theorem~\ref{Aff-impl-strict+DXInv}.

Note that we can assume without loss of generality that $\vecx_0 = \zero$, and so $X$ is affinely independent if and only of $\{\vecx_1,\dots,\vecx_m\}$ is linearly independent.

Suppose then that $X$ is affinely dependent. Then there exist real constants $c_1,\dots,c_m$, not all zero, such that $\sum_{j=1}^m c_j \vecx_j = \zero$. That is, for $1 \le k \le n$, $\sum_{j=1}^m c_j w_k b_{j,k} = 0$. Now let $c_0 = - \sum_{j=1}^m c_j$, let $\vecc = (c_0,c_1,\dots,c_m)^T$ and let
$D_X \vecc = \vecv = (v_0,v_1,\dots,v_m)^T$.

Recall that if $b,b' \in \{0, 1\}$, then $|b-b'| = b - 2bb' + b'$. 
Thus, for $0 \le i \le m$,
  \begin{align*}
  v_i &= \sum_{j=0}^m d(\vecx_i,\vecx_j) c_j \\
  &= d(\vecx_i, \vecx_0) c_0 + \sum_{j=1}^m \sum_{k=1}^n w_k |b_{i,k} - b_{j,k}| c_j \\
  &= \sum_{k=1}^m w_k b_{i,k} c_0
     + \sum_{j=1}^m \sum_{k=1}^n w_k (b_{i,k}-2b_{i,k} b_{j,k} + b_{j,k}) c_j \\
  &= \sum_{k=1}^m w_k b_{i,k} c_0
     + \sum_{j=1}^m \sum_{k=1}^n w_k  b_{i,k} c_j
       + \sum_{k=1}^n (1-2b_{i,k}) \sum_{j=1}^m w_k c_j b_{j,k} \\
  &= \sum_{k=1}^m w_k b_{i,k} c_0
        + \Big(\sum_{j=1}^m c_j\Bigr) \Bigl(\sum_{k=1}^n w_k b_{i,k} \Bigr) \\
  &=(c_0-c_0) \sum_{k=1}^n w_k b_{i,k} \\
  &=0.
  \end{align*}
Thus $D_X \vecc = \zero$. 
It follows that $D_X$ is not invertible, and so $\det(D_X) = 0$.
\end{proof}

We can now give an extension of Theorem~\ref{det-w-tree} from trees to more general affinely independent sets of full dimension in $H_\WS$. 

\begin{theorem}\label{detDX-full-dim} Suppose that $X = \{\vecx_0 = \zero,\vecx_1,\dots,\vecx_n\} \subseteq H_\WS \subseteq \mR^n$ is affinely independent. Then
  \[ \det(D_X) = (-1)^n 2^{n-1} \Bigl(\prod_{k=1}^n w_k\Bigr)
                     \Bigl( \sum_{k=1}^n w_k\Bigr) \det(\hat B)^2.\]
\end{theorem}

\begin{proof}
As $X$ is affinely independent, Lemma~\ref{hatG-lem} implies that both $G_W$ and $B$ are invertible. 
It follows from Lemma~\ref{detDX} and (\ref{det-ident}) that
  \[ \det(D_X) = (-1)^n 2^{n-1} \det(G_W) \ip<G_W^{-1}\vecu,\vecu>. \]
Now $G_W = B\,W^{-1}B = \hatB\, W \hatB^T$ so $\det(G_W) = \Bigl(\prod_{i=k}^n w_k\Bigr) \det(\hat B)^2$.

Now $G_W^{-1} = (B^T)^{-1} W B^{-1}$. Note that the vector $\vecu$ from the previous section can be written as $\vecu = B \one$. Thus
  \begin{align*}
   \ip< G_W^{-1} \vecu, \vecu>
      &= \ip< (B^T)^{-1} W B^{-1} B \one, B \one> \\
      &= \ip<  (B^{-1})^T W \one, B \one> \\
      &= \ip< W \one,\one> \\
      &= \sum_{k=1}^n w_k
  \end{align*}
which completes the proof.
\end{proof}

Let $\hatX$ be an $(n+1)$-element subset of $H_n$. One can `stretch' the hypercube by factors $w_1,\dots,w_n$ to give a subset $X \subseteq H_\WS$. What the previous theorem shows is that the determinant of the distance matrix for $X$ depends only on the shape of the original set $\hatX$ and on the set of edge weights, but not on which factor is applied to which direction. 

\section{Affine independence and strict $1$-negative type}\label{S:Aff-strict}

As noted earlier, every nonempty subset of a Hamming cube is of $1$-negative type.
Murugan \cite{Mur} showed that a subset $X$ of the Hamming cube is of strict $1$-negative type if and only if it is affinely independent as a subset of $\mR^n$. 
Combining Theorems \ref{DXInv-iff-strict} and \ref{AI-iff-invert} allows us to extend Murugan's result to subsets of $H_\WS$.

\begin{theorem}\label{Mur-plus}
A nontrivial subset $X$ of $H_\WS$ is of strict $1$-negative type if and only if it is affinely independent.
\end{theorem}

Theorem~\ref{Mur-plus} completes the proof of our main result, Theorem~\ref{omnibus}. As noted earlier, conditions (3), (4), (5), (6) and (7) are equivalent in the category of finite metric spaces. That (1) and (2) are equivalent for nontrivial subsets of $H_\WS$ is Theorem~\ref{AI-iff-invert}. The equivalence of (2) and (3) is Theorem~\ref{DXInv-iff-strict}. 


\section{The $M$-constant for subsets
of $H_\WS$}\label{S:M-const}

As noted in \cite{NW3}, the $M$-constant of a finite metric space of $1$-negative type contains geometric information. We begin by
recalling the notion of an $S$-embedding.

\begin{definition}
    Let $(X, d)$ be a metric space and let $Y$ be a real or
    complex inner product space with induced norm
    $\norm{\cdot}_{2}$. A map $\iota: X \to Y$ is an
    \textbf{$S$-embedding} if
    $\norm{\iota(x) - \iota(y)}_2 = d(x, y)^{1/2}$
    for all $x, y \in X$.
\end{definition}

The following result is due to Nickolas and Wolf
\cite[Theorem~3.2]{NW3}.

\begin{theorem}\label{min-sphere} 
Let $X = \{x_0,\dots,x_m\}$ be a finite metric space of $1$-negative type and let $\iota: X \to \mR^n$ be an S-embedding of $X$. Suppose that $\iota(X)$ lies on a sphere of radius $r$ with centre $c$ which lies inside the affine hull of $\{\iota(x_0), \dots, \iota(x_m)\}$. Then
  \[ M(X) = 2r^2.\]
\end{theorem}

If $X \subseteq (H_n,\norm{\cdot}_1)$, the identity map is an $S$-embedding and the above result was used in \cite{DWo} to give a formula for $M(X)$ and $\ip<D_X^{-1} \one,\one>$, resolving a conjecture from \cite{DRSW}. A similar formula holds for subsets of $H_\WS$, although in this setting one needs to modify the argument slightly to deal with the issue that the identity map is no longer an S-embedding.

As in Section \ref{S:HammingCubes},
given a weighted Hamming cube $H_{\WS}$ with weight list
$\WS = [w_{1}, w_{2}, \ldots, w_{n}]$, we let
$W = \mathrm{diag}(w_1,\dots,w_n)$.
This allows us to express each
$\vecx \in H_{\WS}$ as
  \[ \vecx = \sum\limits_{k=1}^{n} w_{k}b_{k}\vece_{k} 
      = W \cdot \sum\limits_{k=1}^{n} b_{k}\vece_{k} 
     = W \hatx\]
where each $b_{k} \in \{ 0,1 \}$ and $\vece_{k}$ denotes the $k$th
standard basis vector in $\mathbb{R}^{n}$. One may then easily verify
that the map $\iota : H_{\WS} \rightarrow \mathbb{R}^{n}$ given by
  \[ \iota(\vecx) = \sum\limits_{k=1}^{n} \sqrt{w_{k}}b_{k}\vece_{i} 
             = W^{-\frac{1}{2}} \vecx 
             = W^{\frac{1}{2}} \hatx  \]
is an $S$-embedding. We will call this map the \textbf{natural}
$S$-embedding of $H_{\WS}$ into $\mathbb{R}^{n}$.
Moreover, as a map from $\mathbb{R}^{n}$ to
$\mathbb{R}^{n}$, $\iota$ is linear and invertible.

Let $\WS(\frac{1}{2}) = [\sqrt{w_1},\dots,\sqrt{w_n}]$ so that $H_{\WS(1/2)}$ is the image of $H_\WS$ under $\iota$. The set $H_{\WS(1/2)}$, and hence also $\iota(X)$, lie in a sphere with centre $\vech = \frac{1}{2} \bigl(\sqrt{w_1},\dots,\sqrt{w_n}\bigr)^{T}$ and radius $r$ where
  \[ r^2 = \frac{1}{4} \sum_{k=1}^n w_k. \]
So the issue in applying Theorem~\ref{min-sphere} is whether $\vech$ lies in the affine hull of $\iota(X)$. Of course if the affine hull of $\iota(X)$ is all of $\mR^n$ then this is automatic, as it is if $X$ contains two antipodal points in $H_\WS$. In particular, if $X$ has $n+1$ points and is affinely independent then
  \[ M(X) = \frac{1}{2} \sum_{k=1}^n w_k\]
independent of the actual position of the points of $X$. Specifically, this applies to the case of metric trees.

\begin{proposition}
Suppose that $X$ is an $(n+1)$-point metric tree with edge weights $w_1,\dots,w_n$. Then $M(X) = \tfrac{1}{2}\sum_{k=1}^n w_k$, independent of the structure of the tree and of how the weights are assigned to the edges.  
\end{proposition}

More generally, extending \cite[Theorem~5.1]{DWo} we have the following. Given $X \subseteq H_\WS$ let $Z_X$ denote the (smallest) affine subspace of $\mR^n$ which contains $\iota(X)$. For $\vecx \in \mR^n$, let $d_2(\vecx,Z_X)$ denote the Euclidean distance from $\vecx$ to this subspace. 

\begin{theorem}\label{M-in-HWS}
Suppose that $X \subseteq H_\WS$. Then
  \[ M(X) = \frac{1}{2} \sum_{k=1}^n w_k - 2 d_2(\vech,Z_X)^2. \] 
\end{theorem}

The proof is essentially the same as the one in the Hamming cube case.

\begin{corollary}
    Suppose that $X$ is an affinely independent subset of $H_\WS$. Then
  \[ \ip<D_X^{-1} \one,\one> \ge 2 \Bigl(\sum_{k=1}^n w_k\Bigr)^{-1}, \]
  with equality if and only if $\vech \in Z_X$.
\end{corollary}

\begin{proof}
By Theorem~\ref{omnibus}, $(X,d)$ is of strict $1$-negative type and $D_X$ is invertible. Proposition~\ref{DInv-iff-M_finite} then gives that
  \[ \ip<D_X^{-1} \one,\one>
     = M(X)^{-1} 
     = \Bigl(\tfrac{1}{2}\sum_{k=1}^n w_k - 2 d_2(\vech,Z_X)^2 \Bigr)^{-1}
     \]
from which the conclusion follows.
\end{proof}

We can now extend the result of Graham, Hoffman and Hosoya \cite{GHH} on cofactor sums mentioned in the introduction.
As in Section \ref{S:HammingCubes}, given a set
$X = \{\vecx_0 = \zero,\vecx_1,\dots,\vecx_m\} \subseteq H_\WS$,
$B$ is the $m \times n$ matrix whose $i$th
row is $\vecx_i$ and $\hatB$ is the $m \times n$ matrix
whose $i$th row is $\hatx_i$.

\begin{corollary}
Suppose that
$X = \{\vecx_0 = \zero,\vecx_1,\dots,\vecx_n\}$
is an $(n+1)$-point subset of $H_\WS$. Then
  \[ \cof(D_X) = (-1)^n \, 2^n \Bigl( \prod_{k=1}^n w_k \Bigr) \det(\hatB)^2. \]
\end{corollary}

\begin{proof}
If $X$ is affinely independent then $D_X$ is invertible and so the identity follows immediately from Cramer's rule, 
  $\det(D_X) \ip<D_X^{-1} \one,\one> = \cof(D_X)$, 
and the above formulas.

If $X$ is not affinely independent then $\det(\hatB) = 0$. By Theorem~\ref{AI-iff-invert} we know that $D_X$ is singular.
The standard 
cofactor expansion formula then gives that for each $i$
  \[ 0 = \det(D_X) = \sum_{j=0}^n (-1)^{i+j} 
            \det\bigl((D_X)_{i|j}\bigr)\]
and so summing over $i$ shows that $\cof(X) = 0$ which gives the required result.
\end{proof}


\section{Polygonal equalities admitted in weighted
Hamming cubes}\label{S:Poly-ineq}

 We now consider the problem of characterizing the nontrivial
$1$-polygonal equalities admitted in weighted Hamming cubes
$H_{\WS}$. In order to access these
equalities we begin with several lemmas.

Given a finite set $X = \{ \vecx_{0}, \vecx_{1}, \ldots, \vecx_{m}\}$ in a vector space $Y$, let
  \[ \mathcal{F}_0(X) = \{\bxi = (\xi_0,\xi_1,\dots,\xi_m)^{T}
  \in F_0 \st 
        \xi_{0} \vecx_{0} + \xi_{1}\vecx_{1} + \cdots
                         + \xi_{m}\vecx_{m} = \mathbf{0} \}. \]

\begin{lemma}\label{ip-lem}
    Let $(Y, \ip<\cdot\,, \cdot>)$ be a real inner
    or complex product space
    with induced norm $\norm{\cdot}_{2}$. For each finite subset
    $X = \{ \vecx_{0}, \vecx_{1}, \ldots, \vecx_{m}\} \subseteq Y$
    and each
    vector $\bxi = (\xi_{0}, \xi_{1}, \ldots, \xi_{m})^{T}
    \in F_{0}$
    we have:
    \[
    \ip<D_{X,2}\bxi, \bxi> =
    -2 \snorm{\sum\limits_{i} \xi_{i} \vecx_{i}}_{2}^{2}.
    \]
    Consequently, 
    \[ \mathcal{F}_0(X) 
          = \{ \bxi \in F_0 \st \ip<D_{X,2} \bxi,\bxi> = \zero\} . \]
\end{lemma}

\begin{proof}
    Given such a finite set $X \subseteq Y$ and vector $\bxi \in F_{0}$
    we see that:
    \begin{align*}
 \ip<D_{X,2} \bxi,\bxi> 
    &= \sum_{i,j} \xi_i \xi_j \norm{\vecx_i - \vecx_j}_2^2 \\
    &= \sum_{i,j} \xi_i \xi_j
    \ip<\vecx_i - \vecx_j,\vecx_i - \vecx_j> \\
    &= \sum_j \xi_j \cdot \sum_i \xi_i \norm{\vecx_i}_2^2
       - \sum_{i,j} (\ip<\xi_i \vecx_i, \xi_j \vecx_j>
       + \ip<\xi_j \vecx_j, \xi_i \vecx_i>)
       + \sum_i \xi_i \cdot \sum_j \xi_j \norm{\vecx_j}_2^2  \\
    &= - 2 \ip< \sum_i \xi_i \vecx_i, \sum_j \xi_j \vecx_j> \\
    &= -2 \snorm{\sum_i \xi_i \vecx_i}_2^2.         
  \end{align*} 
\end{proof}

\begin{lemma}\label{s-embed}
    If a finite metric space $(X, d) =
    (\{ x_{0}, x_{1}, \ldots, x_{m} \}, d)$ admits an
    $S$-embedding $\iota : X \rightarrow Y$, where $Y$
    is a real or complex inner product space
    with induced norm $\norm{\cdot}_{2}$, then for each
    vector $\bxi = (\xi_{0}, \xi_{1}, \ldots, \xi_{m})^{T}
    \in F_{0}$
    we have
    \begin{align}
        \ip<D_X \bxi,\bxi>
        &= -2 \snorm{\sum_i \xi_i \iota(x_i)}_2^2.
    \end{align}
\end{lemma}

\begin{proof}
    Let $D_{\iota(X),2}$ denote the $2$-distance matrix of the metric
    subset $\iota(X) \subseteq Y$. We are given that
    $\iota : (X, \sqrt{d}) \rightarrow Y$ is an isometry.
    Consequently, we have $\ip<D_X \bxi,\bxi> = \ip<D_{\iota(X),2} \bxi,\bxi>$.
    The lemma now follows by applying Lemma~\ref{ip-lem}
    to the finite subset $\iota(X) \subseteq Y$.
\end{proof}

Applying Lemma~\ref{ip-lem} and the underlying definitions gives the following.

\begin{lemma}\label{ip-eq}
    Let $(Y, \ip<\cdot\,, \cdot>)$ be a real or complex
    inner product space. The following conditions
    are equivalent for each finite subset
    $X = \{ \vecx_{0}, \vecx_{1}, \ldots, \vecx_{m}\}$ of $Y$:
    \begin{enumerate}[(i)]
        \item $X$ is affinely dependent
        when $Y$ is regarded as a vector space over $\mathbb{R}$.
        \item There exists a nonzero vector
        $\bxi = (\xi_{0}, \xi_{1}, \ldots, \xi_{m})^{T}
        \in F_{0}$ such that
        \[
        \xi_{0} \vecx_{0} + \xi_{1}\vecx_{1} + \cdots
        + \xi_{m}\vecx_{m} = \mathbf{0}.
        \]
        \item There exists a nonzero vector
        $\bxi = (\xi_{0}, \xi_{1}, \ldots, \xi_{m})^{T}
        \in F_{0}$ such that
        $\ip<D_{X,2}\bxi, \bxi> = 0$. In other words, $X$ admits a
        nontrivial $2$-polygonal equality.
        \item $\mathcal{F}_0(X) \ne \{\zero\}$. 
    \end{enumerate}
\end{lemma}

We remark that the equivalence of (i) and (iii) in Lemma
\ref{ip-eq} is originally due to Kelleher et al.\
\cite[Corollary 5.4]{KMOW}.
The proof given above is simpler and more explicit. Moreover,
the equivalences of Lemma \ref{ip-eq}, taken as a whole,
provide a complete description of the nontrivial $2$-polygonal
equalities that can be admitted in real or complex inner
product spaces. One consequence of this description is the
following surprising corollary.

\begin{corollary}
Suppose that $X, Z$ are finite subsets of a real inner product
space $Y$. If $X$ is affinely dependent and $Z$ is affinely
independent, then $2 = \wp(X) < \wp(Z)$.
\end{corollary}

\begin{proof}
This follows from Lemma \ref{ip-eq} and
Li and Weston \cite[Corollary 4.3]{LiW}
\end{proof}

On the basis of Lemmas \ref{ip-lem} through \ref{ip-eq}
we are now able to provide
a complete description of the nontrivial $1$-polygonal
equalities that can be admitted in weighted Hamming cubes.

For the remainder of this section let $\WS = [w_{1}, \ldots, w_{n}]$ denote a set of positive weights, so that $H_\WS \subseteq \mR^n$.

\begin{theorem}\label{1-poly}
    Suppose that $(X = \{\vecx_0,\dots,\vecx_m\},d)$ is a metric subspace of $H_\WS$. Then the following conditions are equivalent for each $\bxi \in F_0$.
    \begin{enumerate}[(i)]
    \item $\sum_{k=0}^m \xi_k \vecx_k = \zero$.
    \item $\ip<D_X \bxi,\bxi> = 0$.
    \item $D_X \bxi = \zero$.
    \end{enumerate}
\end{theorem}

\begin{proof}
If we let $\iota$ denote the natural $S$-embedding of $H_{\WS}$
into $\mathbb{R}^{n}$, then $D_X = D_{\iota(X),2}$.
Furthermore, as a map from $\mathbb{R}^{n}$ to $\mathbb{R}^{n}$,
$\iota$ is linear and invertible. Now let $\bxi \in F_{0}$
be given and apply Lemma \ref{ip-lem} to conclude that
    \[
    \ip<D_{X} \bxi,\bxi> = \ip<D_{\iota(X),2} \bxi,\bxi>
    = -2 \snorm{\sum\limits_{i} \xi_{i}\iota(\vecx_{i})}_{2}^{2}
    = -2 \snorm{\iota
    \biggl(\sum\limits_{i} \xi_{i}\vecx_{i}\biggr)}_{2}^{2}.
    \]
    Thus $\ip<D_{X} \bxi,\bxi> = 0$ if and only if
    $\sum\limits_{i} \xi_{i}\vecx_{i} = 0$.

Suppose now that $\sum\limits_{i} \xi_{i}\vecx_{i} = 0$. For $0 \le i \le m$ and $1 \le k \le n$ let $x_{i,k} = w_k b_{i,k}$ denote the $k$th component of $\vecx_i$. Then, for $1 \le k \le n$
  \[ \sum_{j=0}^m x_{j,k} \, \xi_j 
       = \sum_{j=0}^m w_k \,b_{j,k} \,\xi_j = 0. \]
Using the properties of binary digits as in the proof of Theorem~\ref{AI-iff-invert}, we see that if $0 \leq i \leq m$,
the $i$th component of $D_{X}\bxi$ is given by
\begin{eqnarray*}
  \sum\limits_{j = 0}^{m}
    \norm{ \vecx_{i} - \vecx_{j} }_{1} \xi_{j}  
    & = & \sum\limits_{j = 0}^{m} \sum\limits_{k = 1}^{n}
                    w_k |b_{i,k} - b_{j,k}| \,\xi_{j} \\
    & = & \sum\limits_{j = 0}^{m}
    \left( \sum\limits_{k = 1}^{n} w_k
    \big(b_{i,k} + b_{j,k} - 2 b_{i,k} b_{j,k}\big)
    \right) \xi_{j} \\
    & = & \sum\limits_{j = 0}^{m} \xi_{j} \left(
    \sum\limits_{k = 1}^{n} w_k b_{i,k} \right) +
    \sum\limits_{k = 1}^{n} \left( \sum\limits_{j = 0}^{m}
     w_k  \, b_{j,k} \, \xi_{j} \right) 
      - 2 \sum\limits_{k = 1}^{n}
    \left( b_{i,k}  
       \sum\limits_{j = 0}^{m} w_k \, b_{j,k} \, \xi_{j} \right) \\
    & = & 0.
\end{eqnarray*}
As this is true for all $i$ we conclude
that $D_{X}\bxi = \zero$.

Clearly if $D_X \bxi = \zero$ then $\ip<D_X\bxi,\bxi> = 0$ so this completes the proof.
\end{proof}

\begin{corollary}\label{s-san-i}
    Suppose that $(X,d)$ is a
    metric subspace of $H_\WS$. Then
    $\ker (D_{X}) = \mathcal{F}_{0}(X)$.
\end{corollary}

\begin{proof}
    By Theorem \ref{1-poly}, we have $\mathcal{F}_{0}(X)
    \subseteq \ker (D_{X})$. Since obviously
    $\bxi = \zero \in \mathcal{F}_{0}(X)$, it follows
    that $\ker (D_{X}) = \{ \zero \}$ if and only if
    $\mathcal{F}_{0}(X) = \{ \zero \}$.

    Now suppose that $\ker (D_{X})$ is nontrivial.
    In other words, assume that $\det(D_{X}) = 0$.
    Then $X$ is affinely dependent by Theorem
    \ref{AI-iff-invert}. This entails that $\wp(X) = 1$
    by Theorem \ref{Mur-plus}.
    Hence $\ker (D_{X}) \subseteq F_{0}$ by S\'{a}nchez
    \cite[Corollary 2.5 (i)]{San}. In particular,
    each $\bxi \in \ker (D_{X})$ satisfies
    $\sum_{k=0}^m \xi_k \vecx_k = \zero$ by Theorem
    \ref{AI-iff-invert}. So, in fact,
    $\ker (D_{X}) \subseteq \mathcal{F}_{0}(X)$.
    This completes the proof.
\end{proof}

\begin{remark}\label{s-san-ii}
In the setting of Corollary \ref{s-san-i}, if $\wp(X) = 1$,
then $\det (D_{X}) = 0$ (and so $\ker(D_X)$ is nontrivial)
by Theorems \ref{Mur-plus} and \ref{AI-iff-invert}
(in that order). In particular, each nonzero
$\bxi \in \ker(D_X)$ then provides a nontrivial $1$-polygonal
equality $\ip<D_X \bxi,\bxi> = 0$ by Theorem \ref{1-poly}.

In contrast, as the following example shows, one can have $\sum_{k=0}^m \xi_k \vecx_k = \zero$ or $\ip<D_X \bxi,\bxi> = 0$ with $\bxi \not\in F_0$. 
\end{remark}

\begin{example}
Let the set
  \[ X = \{ \vecx_0,\vecx_1,\vecx_2,\vecx_3\} = \left\{ \binom{0}{0}, \binom{1}{0}, \binom{0}{2}, 
            \binom{1}{2} \right\} \subseteq \mR^2 \]
be endowed with the $\ell^1$ metric.
To see that one can't drop the 
hypothesis that $\bxi \in F_0$ in the theorem observe that if $\bxi = (1,0,0,0)^T$, then $\sum_k \xi_k \vecx_k = \zero$ but $D_X \bxi \ne \zero$. If $\bxi = (1,-1,0,1)^T$, then $\ip<D_X \bxi,\bxi> = 0$ but $\sum_k \xi_k \vecx_k \ne \zero$ and $D_X \bxi \ne \zero$.
\end{example}

\begin{example}\label{Ex6.7}
Theorem~\ref{1-poly} is dependent on the structure of $H_\WS$
and does not extend to more general subsets of $(\mR^n,\norm{\cdot}_1)$.
To see this let the set
  \[ X = \{ \vecx_0,\vecx_1,\vecx_2,\vecx_3,\vecx_4\} = \left\{ \binom{0}{0}, \binom{1}{0}, \binom{0}{2}, 
            \binom{1}{2}, \binom{2}{2} \right\} \subseteq \mR^2\]
be endowed with the $\ell^1$ metric.
The kernel of $D_X$ is $\mathrm{span}((1,-1,-1,1,0)^T) \subseteq F_0$. We note that $\bxi = (2,-2,-1,0,1)^T$ is an element of $F_0$ for which $\sum_k \xi_k \vecx_k = \zero$ but for which $D_X \bxi \ne \zero$ and $\ip<D_X\bxi,\bxi> \ne 0$. 
\end{example}

\begin{example} If $p > 1$ the set of $\bxi \in F_0$ for which $\ip<D_p \bxi,\bxi> = 0$ is in general not a linear subspace. Let $X$ be the star graph on 4 vertices,
  \[ X = \left\{ \tri(0,0,0) , \tri(1,0,0) \tri(0,1,0), \tri(0,0,1) 
           \right\} \subseteq H_3. \]
It is known that $\wp(X) = \frac{\ln 3}{\ln 2} \approx 1.585$. Taking $p = 2 > \wp(X)$ one finds that $\ip<D_2 \bxi,\bxi> = 0$ if and only if
   $ \bxi = (-2s -2t  \pm 2\sqrt{st},s,t,s+t \mp 2\sqrt{st})^T $ 
with $st \ge 0$, which does not form a  subspace of
$\mR^4$.
\end{example}

Under the hypotheses of Theorem~\ref{omnibus}, we saw that $D_X$ is invertible if and only if $X = \{\vecx_0,\dots,\vecx_m\}$ is affinely independent. This can be rephrased as saying that $\rank(D_X) = m+1$ if and only if $\dimsp(\vecx_1-\vecx_0,\dots,\vecx_m-\vecx_0) = m$. Theorem~\ref{1-poly} allows us to generalize this to provide a quantitative relationship between the rank of $D_X$ and the linear algebraic structure of $X$.

\begin{theorem}\label{rank-thm}
Suppose that $X = \{\vecx_0,\vecx_1,\dots,\vecx_m\} \subseteq H_\WS$ with the usual $\ell^{1}$ metric.  Then
  \begin{equation}\label{rank-ident}
    \mathrm{rank}(D_X) 
         = \dimsp(\vecx_1-\vecx_0,\dots,\vecx_m-\vecx_0) +1. 
  \end{equation}
\end{theorem}

\begin{proof} Both sides of (\ref{rank-ident}) are invariant under the reflection maps considered in Section~\ref{S:HammingCubes}, so without loss of generality we may assume that $\vecx_0 = \zero$.
Let $A$ be the $n \times (m+1)$ matrix whose $j$th column is $\vecx_j$ so that the right-hand side of (\ref{rank-ident}) is $\rank(A)+1$.
Since $D_X$ and $A$ both have $m+1$ columns, it will suffice then to show that $\nullity(D_X) + 1 = \nullity(A)$. Since the first standard basis vector $\vece_1$ lies in $\ker(A)$ we can write $\ker(A) = \mathrm{span}(\vece_1) \oplus \mathcal{N}$ where $\mathcal{N} = \{\bxi = (\xi_0,\dots,\xi_m)^{T} \in
\ker(A) \st \xi_0 = 0\}$. 

For $\bxi \in \ker(D_X)$ let $U \bxi = \bxi - \ip<\bxi,\vece_1> \vece_1 = (0,\xi_1,\dots,\xi_m)^{T}$. Then, by
Theorem~\ref{1-poly},
  \[ A U \bxi = A \bxi - \ip<\bxi,\vece_1> A \vece_1 = \zero \]
and so $U\bxi \in \mathcal{N}$. Conversely, for $\bxi \in \mathcal{N}$ let $V \bxi = \left(-\sum\limits_{i=1}^m \xi_i \right) \vece_1 + \bxi = \bigl(-\sum\limits_{i=1}^m \xi_i,\xi_1,\dots,\xi_m \bigr)^{T}$. Then
$V \bxi \in F_0$ and 
  \[ A V\bxi = \left(-\sum\limits_{i=1}^m \xi_i \right) A \vece_1 - A \bxi = \zero \]
so applying Theorem~\ref{1-poly} again we see that $D_X V\bxi = \zero$. 

It is easy to check that $VU$ and $UV$ are just the identity maps on $\ker(D_X)$ and $\mathcal{N}$ and so these two vector spaces are isomorphic. Thus $\nullity(A) = 1 + \dim(\mathcal{N}) = 1+\nullity(D_X)$, as required.
\end{proof} 

Again this surprising result is rather specific to weighted Hamming cubes. Note that in particular it shows that for subsets of $H_\WS \subseteq \mR^n$, the rank of the distance matrix is never larger than $n+1$. For the subset of $\mR^2$ given in Example~\ref{Ex6.7}, $\rank(D_X) = 4$ and this shows that Theorem~\ref{rank-thm} can not be extended to $(\mR^2,\norm{\cdot}_1)$.

\begin{theorem}\label{affine-S}
    Let $(X,d) =
    (\{ \vecx_{0}, \vecx_{1}, \ldots , \vecx_{m} \}, d)$ be 
    a metric subset of $H_{\WS}$.
    Let $Y$ be a real or complex inner product space.
    Let $\iota : H_{\WS} \rightarrow Y$ be an $S$-embedding.
    Then $X$ is affinely dependent if and only if $\iota(X)$
    is affinely dependent.
\end{theorem}

\begin{proof}
    Combining the previous results,
    \begin{align*}
        \text{$X$ is affinely dependent}
          & \iff \text{$\sum_i \xi_i \vecx_i = \zero$ for some nonzero $\bxi \in F_{0}$} \\
          & \iff \text{$\ip<D_X \bxi,\bxi> = 0$ for some
          nonzero $\bxi \in F_{0}$ 
           \quad (by Theorem~\ref{1-poly})} \\
          & \iff \text{$\sum_i \xi_i \iota(\vecx_i) = \zero$ for some nonzero $\bxi \in F_{0}$ \quad
          (by Lemma~\ref{s-embed})} \\
          & \iff \text{$\iota(X)$ is affinely dependent}.
    \end{align*}
\end{proof}

In general, 
the sets $\mathcal{F}_0(X)$ and $\mathcal{F}_0(\iota(X))$ may be distinct.
However, in the case of
the natural $S$-embedding
$\iota : H_{\WS} \rightarrow \mathbb{R}^{n}$,
these sets coincide. In general, we have the following
theorem.

\begin{theorem}\label{s-embedding}
    Let $(X,d) =
    (\{ \vecx_{0}, \vecx_{1}, \ldots , \vecx_{m} \}, d)$ be 
    a metric subset of $H_{\WS}$.
    Let $Y$ be a real or complex inner product space.
    Let $\iota : H_{\WS} \rightarrow Y$ be an $S$-embedding.
    Then
    \[ \mathcal{F}_0(\iota(X)) 
        = \{\bxi \in F_0 \st \ip<D_X \bxi,\bxi> = \zero\} 
         = \ker(D_X) = \mathcal{F}_0(X) . \]
\end{theorem}

\begin{proof}
    As $\iota$
    is an $S$-embedding, we have
    $\ip<D_{X} \bxi,\bxi> = \ip<D_{\iota(X),2} \bxi,\bxi>$
    for each $\bxi \in F_{0}$.
    Now apply Lemma~\ref{ip-lem} to $\iota(X)$ as a
    metric subset of $Y$ to give the first equality.
    The remaining equalities follow from Corollary
    \ref{s-san-i} and Remark \ref{s-san-ii}.
\end{proof}

We remark that in the setting of Theorem \ref{s-embedding}
one may also conclude that
$\mathcal{F}_0(X) = \mathcal{F}_0(\hatX)$.

We complete this section by noting an application of Lemma
\ref{ip-lem} to the determination of $2$-negative type gaps
in inner product spaces.

\begin{theorem}
    Let $(Y, \ip<\cdot\,, \cdot>)$ be a real or complex
    inner product space. For each finite subset
    $X = \{ \vecx_{0}, \vecx_{1}, \ldots, \vecx_{m}\} \subseteq Y$
    we have:
    \[
    \Gamma_{2}(X) = \min \left\{ 4 \cdot \snorm{\sum\limits_{i} \xi_{i}\vecx_{i}}_{2}^{2}
    : \bxi = (\xi_{0}, \xi_{1}, \ldots, \xi_{m})^{T}
    \in F_{0} \text{ and } \| \bxi \|_{1} = 1 \right\}.
    \]
\end{theorem}

\begin{proof}
    This follows from (\ref{p-gap}) and Lemma \ref{ip-lem}
    by a simple rearrangement.
\end{proof}


\section{Final remarks and open problems}\label{S:Examples}

We have seen that many of the implications included in Theorem~\ref{omnibus} are in fact valid in greater generality than the setting of weighted Hamming cubes, so it is natural to ask whether one might be able to relax the hypothesis in that theorem.
One might ask, for example, whether the theorem is valid for finite subsets of $(\mR^n,\norm{\cdot}_1)$.

Noting that conditions (3), (4), (5), (6) and (7) are equivalent in any finite metric spaces, our aim in this section is to clarify the relationships between conditions (1), (2), (3) in wider classes of spaces than weighted Hamming cubes, namely, finite subsets of $(\mR^n,\norm{\cdot}_1)$, finite spaces of $1$-negative type, and general finite metric spaces. We shall examine here each of the 3 pairs of conditions. (Of course, Condition~(1) only makes sense in the first of these wider classes.)

\subsection{(2) invertibility of $D_X$ and (3) strict $1$-negative type}

We saw in Theorem~\ref{DXInv-iff-strict} that these two conditions are equivalent for finite subsets of $(\mR^n,\norm{\cdot}_1)$. Indeed it follows from S{\'a}nchez's formula that (3) implies (2) in any finite metric space. On the other hand, Example~\ref{Inv-not-strict} shows that there are spaces of $1$-negative type in which Condition~(2) does not imply Condition~(3).

\subsection{(1) Affine independence and (3) strict $1$-negative type}\label{(1)&(3)}
That Condition~(1) implies Condition~(3) for finite subsets of $(\mR^n,\norm{\cdot}_1)$ is \cite[Corollary~4.10]{KMOW}. To see that the converse implication fails for such spaces, consider 
  \[ X = \left\{ \binom{0}{0}, \binom{1}{0}, \binom{1}{2}, \binom{2}{2} \right\} \subseteq (\mR^2,\norm{\cdot}_1). \]
Then $X$ is certainly affinely dependent but
  \[ D_X = \begin{pmatrix}
      0 & 1 & 3 & 4 \\
      1 & 0 & 2 & 3 \\
      3 & 2 & 0 & 1 \\
      4 & 3 & 1 & 0
  \end{pmatrix}\]
which is invertible. Further, $\ip<D_X^{-1} \one,\one> = \frac{506}{4727} > 0$. Noting that $(X,d)$ is of $1$-negative type,
by S{\'a}nchez's formula, we see that the supremal $p$-negative type can't be $1$, and so $(X,d)$ is of strict $1$-negative type. That is, this space satisfies Condition~(3), but not
Condition~(1).

\subsection{(1) Affine independence and (2) invertibility of $D_X$}
Since Conditions (2) and (3) are equivalent for finite subsets of $(\mR^n,\norm{\cdot}_1)$ it follows from the statements in Subsection~\ref{(1)&(3)} that Condition~(1) implies Condition~(2) in this setting, but that that converse implication fails.

There remain a few questions which the authors have been unable to answer. If $(X,d)$ is an $(n+1)$-point metric tree with integer edge weights $\WS = [w_1,\dots,w_n]$, then (from \ref{1-gap}) the $1$-negative type gap satisfies
  \[ \frac{\prod_{k=1}^n w_k }{\Gamma_1(X)}
      = \Bigl(\prod_{k=1}^n w_k\Bigr) \Bigl(\sum_{k=1}^n w_k^{-1} \Bigr) \]
which is always an integer whose value clearly only depends on the weights and not the structure of the tree. The right-hand side of this equation is half the sum of the $(n-1)$-dimensional volumes of the faces of $H_\WS$ and so
  \[ \frac{\Gamma_1(X)}{2} = \frac{\text{Volume of $H_\WS$}}{\text{Volume of faces of $H_\WS$}}.
  \]

Numerical experiments indicate that if $X$ is an affinely independent $(n+1)$-point subset of $H_\WS$  and the weights are positive integers then (using the notation from Section~\ref{S:HammingCubes})
  \[ \frac{\bigl(\prod_{k=1}^n w_k \bigr) \det(\hatB)^2}{\Gamma_1(X)} \]
is also always an integer (and one needs the determinant factor for this to be true). It would be interesting to attach some meaning to this quantity and thereby obtain a geometric formula for the $1$-negative type gap in this setting.

In Section~\ref{S:Background} we recalled that if $(X, d)$ is
a finite metric space with $\wp(X)$ finite, then
\begin{eqnarray}\label{interval}
\{p > 0 \st \text{$(X,d)$ admits a nontrivial $p$-polygonal equality}\} & = & [\wp(X),\infty).
\end{eqnarray}
In Section~\ref{S:Poly-ineq} we showed that a subset $X$ of
a weighted Hamming cube $H_\WS$
admits a nontrivial $1$-polygonal equality
if and only if $X$ is affinely dependent. For each $p > 1$ it
remains an open problem to characterize the subsets of $H_\WS$
that admit a nontrivial $p$-polygonal equality.
By \ref{interval} and Theorem~\ref{1-poly},
if $p > 1$ and if $X \subseteq H_\WS$ is
affinely dependent, then $X$ admits a nontrivial $p$-polygonal
equality. The converse will also be true if $p>1$ is sufficiently
close to $1$. This is simply because $H_\WS$ has only finitely
many
subsets. There will also exist a smallest $p \in (1, \infty)$
such that every non-ultrametric metric subspace of $H_\WS$
admits a nontrivial $p$-polygonal equality. Necessarily,
this value of $p$ will equal $\wp(X)$ for some set
$X \subseteq H_\WS$.

More generally, if we let $\mathcal{X}_{p}$ denote the set
of subsets of $H_\WS$ that admit a nontrivial $p$-polygonal
equality, there will exist finitely many real numbers
$1 = p_{0} < p_{1} < \dots < p_{l}$ such that
$\mathcal{X}_{p_{0}}, \ldots, \mathcal{X}_{p_{l}}$ are
pairwise distinct sets and
\[
\mathcal{X}_{p} = \begin{cases}
\mathcal{X}_{p_{k-1}} & \text{ if } p \in [p_{k - 1}, p_{k}) \\
\mathcal{X}_{p_{l}}   & \text{ if } p \in [p_{l}, \infty).
\end{cases}
\]
By \ref{interval}, we will have
$\mathcal{X}_{p_{0}} \subseteq \mathcal{X}_{p_{1}} \subset
\ldots \subseteq \mathcal{X}_{p_{l}}$, with all inclusions strict.
The determination of the real numbers $p_{1}, \ldots, p_{l}$
and characterization of the corresponding sets
$\mathcal{X}_{p_{1}}, \ldots, \mathcal{X}_{p_{l}}$
are open problems, the resolution of which would provide
deep insights into the metric geometry of weighted Hamming
cubes.

\bibliographystyle{amsalpha}

\begin{thebibliography}{ABC}

\bibitem{ALN} S. Arora, J. R. Lee and A. Naor, 
     \textit{Euclidean distortion and the sparsest cut},
    J. Amer. Math. Soc. \textbf{21} (2008), 1--21.

\bibitem{AS} R. Alexander and K. B. Stolarsky,
     \textit{Extremal problems of distance geometry related to energy integrals},
     Trans. Amer. Math. Soc. \textbf{193} (1974), 1--31.

\bibitem{BKN} R. Bapat, S. J. Kirkland and M. Neumann,
      \textit{On distance matrices and Laplacians},
      Linear Algebra Appl. \textbf{401} (2005), 193--209.

\bibitem{BSetc} J. Berleant, K. Sheridan, A. Condon, V. Vassilevska Williams and M. Bathe, Mark,
     \textit{Isometric Hamming embeddings of weighted graphs},
     Discrete Appl. Math. \textbf{332} (2023), 119--128.

\bibitem{CGR} S. Chawla, A. Gupta and H. R{\"a}cke,
   \textit{Embeddings of negative-type metrics and an improved approximation to generalized sparsest cut},
   ACM Trans. Algorithms \textbf{4} (2008), Art. 22, 18 pp.

\bibitem{CK} P. N. Choudhury and A. Khare,
    \textit{Distance matrices of a tree: Two more invariants, and a unified framework},
    European J. Combin. \textbf{115} (2024), Paper No. 103787, 30pp.

\bibitem{DRSW} I. Doust, G. Robertson, A. Stoneham and A. Weston,   
     \textit{Distance matrices of subsets of the Hamming cube}, 
   Indag. Math. \textbf{32} (2021), 646--657.
              
\bibitem{DWe} I. Doust and A. Weston,
  \textit{ Enhanced negative type for finite metric trees}, 
    J. Funct. Anal. \textbf{254} (2008), 2336--2364; with Corrigendum.

\bibitem{DWe2} I. Doust and A. Weston,
  \textit{Polygonal equalities and $p$-negative type},
    arXiv.
              
\bibitem{DWo} I. Doust and R. Wolf, 
     \textit{A problem on distance matrices of subsets of the Hamming cube},
      Proc. Amer. Math. Soc. Ser. B \textbf{9} (2022), 125--134.

\bibitem{EHKMZ}
  L. Elsner, L.  Han, I. Koltracht, M. Neumann, M. Zippin,
  \textit{On a polygon equality problem},
  J. Math. Anal. Appl. \textbf{223} (1998),  67--75.
      
\bibitem{En1} P. Enflo, 
    \textit{On a problem of Smirnov}, 
     Ark. Mat. \textbf{8} (1969), 107--109.

\bibitem{Fav} T. Faver, K. Kochalski, M. Murugan, H. Verheggen, E. Wesson, A. Weston, 
    \textit{Roundness properties of ultrametric spaces},
     Glasgow Math. J. \textbf{56} (2014), 519--535.

\bibitem{GHH} R. L. Graham, A. J. Hoffman and H. Hosoya,
   \textit{On the distance matrix of a directed graph}.
   J. Graph Theory \textbf{1} (1977), 85--88.

\bibitem{GP} R. L. Graham and H. O. Pollak,
    \textit{On the addressing problem for loop switching},
    Bell System Tech. J. \textbf{50} (1971), 2495--2519.

\bibitem{GW} R. L. Graham and P. M.  Winkler,
    \textit{On isometric embeddings of graphs},
     Trans. Amer. Math. Soc. \textbf{288} (1985), 527--536.

\bibitem{HLMT} 
  P. Hjorth, P. Lison\u{e}k, S. Markvorsen and C. Thomassen, 
    \textit{Finite metric spaces of strictly negative type},
    Linear Algebra Appl. \textbf{270} (1998), 255--273.

\bibitem{KMOW} 
  C. Kelleher, D. Miller, T. Osborn, A. Weston, 
   \textit{Polygonal equalities and virtual degeneracy in $L_p$-spaces},
   J. Math. Anal. Appl. \textbf{415} (2014),  247--268.

\bibitem{LTW} C. J. Lennard, A. M. Tonge and A. Weston, 
   \textit{Generalized roundness and negative type},
              Michigan Math. J. \textbf{44} (1997), 37--45.

\bibitem{LiW} H. Li and A. Weston,
   \textit{Strict $p$-negative type of a metric space},
   Positivity \textbf{14} (2010), 529--545.

\bibitem{Lin} N. Linial,
\textit{Finite metric spaces--combinatorics, geometry and algorithms}, Proceedings of the International Congress of Mathematicians, Vol. III (Beijing, 2002) Higher Ed. Press, Beijing, 2002, pp. 573–586.   

\bibitem{Mur} M. K. Murugan, 
   \textit{Supremal $p$-negative type of vertex transitive graphs},
   J. Math. Anal. Appl. \textbf{391} (2012), 376--381.

\bibitem{NWF} P. Nickolas and R. Wolf, 
   \textit{Finite quasihypermetric spaces},
   Acta Math. Hungar. \textbf{124} (2009), 243--262. 

\bibitem{NIC} P. Nickolas and R. Wolf, \textit{Distance geometry and quasihypermetric spaces. I.},
     Bull. Aust. Math. Soc. \textbf{80} (2009), 1--25.
              
\bibitem{NW3} P. Nickolas and R. Wolf, \textit{Distance geometry and quasihypermetric spaces. III.},
	Math. Nachr. \textbf{284} (2011), 747--760.

\bibitem{RoW} R. L. Roth and P. M. Winkler, 
   \textit{Collapse of the metric hierarchy for bipartite graphs}, 
    European J. Combin. (1986) \textbf{7}, 371--375.   

\bibitem{San} S. S{\'a}nchez,
    \textit{On the supremal $p$-negative type of a finite metric space},
    J. Math. Anal. Appl., \textbf{389} (2012), 98--107.

\bibitem{Sc3} I. J. Schoenberg, 
   \textit{Metric spaces and positive definite functions}, 
   Trans. Amer. Math. Soc. \textbf{44} (1938), 522--536.

\bibitem{W2017} A. Weston,
      \textit{The geometry of two-valued subsets of $L_p$-spaces},
      Math. Slovaca \textbf{67} (2017), 751--758.
      
\bibitem{Wo} R. Wolf, 
     \textit{On the gap of finite metric spaces of $p$-negative type}, Linear Algebra Appl. \textbf{436} (2012), 1246--1257.
              
\bibitem{ZD} H. Zhou and Q. Ding, 
    \textit{The distance matrix of a tree with weights on its arcs},
    Linear Algebra Appl. \textbf{511} (2016), 365--377.


\end{thebibliography}

\end{document}